\documentclass[12pt]{amsart}
\usepackage{amscd,amsmath,amsthm,amssymb}
\usepackage{pstcol,pst-plot,pst-3d}
\usepackage{lmodern,pst-node}
\usepackage{lineno}
\def\NZQ{\Bbb}               

\def\ZZ{{\NZQ Z}}

\def\FF{{\NZQ F}}
\def\GG{{\NZQ G}}

\def\frk{\frak}               

\def\mm{{\frk m}}

\def\nn{{\frk n}}

\def\opn#1#2{\def#1{\operatorname{#2}}} 
\opn\chara{char} \opn\length{\ell} \opn\pd{pd} \opn\rk{rk}
\opn\projdim{proj\,dim} \opn\injdim{inj\,dim} \opn\rank{rank}
\opn\depth{depth} \opn\grade{grade} \opn\height{height}
\opn\embdim{emb\,dim} \opn\codim{codim}

\opn\Tr{Tr} \opn\bigrank{big\,rank}
\opn\superheight{superheight}\opn\lcm{lcm}
\opn\trdeg{tr\,deg}
\opn\reg{reg} \opn\lreg{lreg} \opn\ini{in} \opn\lpd{lpd}
\opn\size{size}\opn\bigsize{bigsize}
\opn\cosize{cosize}\opn\bigcosize{bigcosize}
\opn\sdepth{sdepth}\opn\sreg{sreg}
\opn\link{link}\opn\fdepth{fdepth}
\opn\Deg{Deg}
\opn\div{div} \opn\Div{Div} \opn\cl{cl} \opn\Cl{Cl}
\opn\Spec{Spec} \opn\Supp{Supp} \opn\supp{supp} \opn\Sing{Sing}
\opn\Ass{Ass} \opn\Min{Min}\opn\Mon{Mon} \opn\dstab{dstab} \opn\astab{astab}
\opn\Ann{Ann} \opn\Rad{Rad} \opn\Soc{Soc}
\opn\Im{Im} \opn\Ker{Ker} \opn\Coker{Coker} \opn\Am{Am}
\opn\Hom{Hom} \opn\Tor{Tor} \opn\Ext{Ext} \opn\End{End}
\opn\Aut{Aut} \opn\id{id}

\opn\nat{nat}
\opn\pff{pf}
\opn\Pf{Pf} \opn\GL{GL} \opn\SL{SL} \opn\mod{mod} \opn\ord{ord}
\opn\Gin{Gin} \opn\Hilb{Hilb}\opn\sort{sort}
\opn\aff{aff} \opn\con{conv} \opn\relint{relint} \opn\st{st}
\opn\lk{lk} \opn\cn{cn} \opn\core{core} \opn\vol{vol}
\opn\link{link} \opn\star{star}\opn\lex{lex} \opn\sat{sat}
\opn\gr{gr}

\def\pot#1#2{#1[\kern-0.28ex[#2]\kern-0.28ex]}

\opn\dirlim{\underrightarrow{\lim}}
\opn\inivlim{\underleftarrow{\lim}}
\let\sect=\cap
\let\dirsum=\oplus

\let\iso=\cong
\let\Union=\bigcup

\let\Dirsum=\bigoplus

\let\to=\rightarrow
\let\To=\longrightarrow
{}
\def\Implies{\ifmmode\Longrightarrow \else
        \unskip${}\Longrightarrow{}$\ignorespaces\fi}
\def\implies{\ifmmode\Rightarrow \else
        \unskip${}\Rightarrow{}$\ignorespaces\fi}
\def\iff{\ifmmode\Longleftrightarrow \else
        \unskip${}\Longleftrightarrow{}$\ignorespaces\fi}

\let\:=\colon
\newtheorem{Theorem}{Theorem}[section]
\newtheorem{Lemma}[Theorem]{Lemma}
\newtheorem{Corollary}[Theorem]{Corollary}
\newtheorem{Proposition}[Theorem]{Proposition}
\newtheorem{Remark}[Theorem]{Remark}

\let\epsilon\varepsilon
\let\kappa=\varkappa
\def\qed{\ifhmode\textqed\fi
      \ifmmode\ifinner\quad\qedsymbol\else\dispqed\fi\fi}
\def\textqed{\unskip\nobreak\penalty50
       \hskip2em\hbox{}\nobreak\hfil\qedsymbol
       \parfillskip=0pt \finalhyphendemerits=0}
\def\dispqed{\rlap{\qquad\qedsymbol}}

\opn\dis{dis}
\def\pnt{{\raise0.5mm\hbox{\large\bf.}}}

\opn\Lex{Lex}

\begin{document}

\title {Koszul binomial edge ideals}
\author {Viviana Ene, J\"urgen Herzog and Takayuki Hibi}

\address{Faculty of Mathematics and Computer Science, Ovidius University, Bd.\ Mamaia 124,
 900527 Constanta,  and
 \newline
 \indent Simion Stoilow Institute of Mathematics of Romanian Academy, Research group of the project  ID-PCE-2011-1023,
 P.O.Box 1-764, Bucharest 014700, Romania} \email{vivian@univ-ovidius.ro}

\address{J\"urgen Herzog, Fachbereich Mathematik, Universit\"at Duisburg-Essen, Campus Essen, 45117
Essen, Germany} \email{juergen.herzog@uni-essen.de}

\address{Takayuki Hibi, Department of Pure and Applied Mathematics, Graduate School of Information Science and Technology,
Osaka University, Toyonaka, Osaka 560-0043, Japan}
\email{hibi@math.sci.osaka-u.ac.jp}

\thanks{The  first author was supported by the grant UEFISCDI,  PN-II-ID-PCE- 2011-3-1023.}

\begin{abstract}
It is shown that if  the binomial edge ideal of a graph $G$ defines a  Koszul algebra, then $G$ must be chordal and claw free. A converse of this statement is proved for a class  of chordal and claw free graphs.
\end{abstract}

\subjclass[2010]{13C13, 13A30, 13F99,  05E40}
\keywords{Koszul algebra, binomial edge ideals }

\maketitle

\section*{Introduction}
A Koszul algebra in our context will be a standard graded (commutative) $K$-algebra whose graded maximal ideal has a linear resolution. This class of $K$-algebras occurs quite frequently  among  toric rings and other $K$-algebras arising in combinatorial commutative algebra and algebraic geometry.  It is known and easily seen that a Koszul algebra is defined by quadrics. This statement has a partial converse, which says that a $K$-algebra is Koszul if its defining ideal admits a reduced  Gr\"obner basis of quadrics. The proof of these statements can for example be found in \cite{EH}.

In the present paper we consider $K$-algebras defined by binomial edge ideals. Given a finite simple graph $G$ on the vertex set $[n]=\{1,2\ldots,n\}$, one defines the binomial edge ideal $J_G$ associated with $G$ as the ideal generated by the quadrics $f_{ij}= x_iy_j-x_jy_i$ in $S=K[x_1,\ldots,x_n, y_1,\ldots,y_n]$  with $\{i,j\}$ an edge of $G$.

This class of ideals was  introduced in \cite{HHHKR} and \cite{Oht}. Part of  the motivation to consider such ideals arises from algebraic statistic as explicated in \cite{HHHKR}, see also \cite{EH}. In recent years several papers appeared (\cite{CR}, \cite{EHH}, \cite{EZ}, \cite{MM}, \cite{MK1}, \cite{MK2}) attempting to describe algebraic and homological properties of binomial edge ideals in terms of the underlying graph. Since by its  definition $J_G$ is generated by quadrics it is natural to ask  for which graphs $G$ the $K$-algebra $S/J_G$ is Koszul. If this happens to be the case we call $G$ Koszul with respect to $K$. As noted above, $G$ will be Koszul if $J_G$ has a quadratic Gr\"obner basis. This is the case with respect to the  lexicographic order induced by $x_1>\cdots >x_n>y_1>\cdots >y_n$ if and only if $G$ is a closed graph  with respect to the given
labeling, in other words,  if $G$  satisfies the following condition: whenever $\{i, j\}$and $\{i, k\}$
are edges of $G$  and either $i < j$,  $i < k$  or $i > j$,  $i > k$ then $\{j,k\}$  is also an edge of $G$.
One calls a graph $G$ closed if it is closed with respect to some labeling of its vertices.
 It was observed  in \cite{HHHKR} that a closed graph must be chordal and claw free. However the class of closed graphs is much smaller than that of chordal and claw free graphs.  Interesting combinatorial characterizations of closed graphs are given in \cite{CE} and \cite{EHH}.

By what we said  so far it follows that all closed graphs are Koszul. On the other hand, it is not hard to find non-closed graphs that are Koszul. Thus the problem  arises to classify all Koszul graphs. In Section~\ref{one} we show that Koszul graphs must be closed and claw free. Thus we have the implications
\[
\text{closed graph \implies Koszul graph \implies chordal and  claw free graph.}
\]
The first implication cannot be  reversed. In Section 2 we give an example of a graph which is chordal and claw free but not Koszul. Thus the second implication cannot be reversed as well.  The results that we have so far allow a classification of all Koszul graphs whose cliques are of dimension at most 2.

\section{Koszul graphs are chordal and claw free}
\label{one}
The goal of this section is to prove the statement made in the section title. We first recall some concepts from graph theory. Let $G$ be a finite simple graph,  that is, a graph with no loops or multiple edges. We denote by $V(G)$ the set of vertices and by $E(G)$  the set of edges of  $G$. A {\em cycle} $C$ of $G$ of length $n$ is a  subgraph of $G$ whose vertices $V(C)=\{v_1,\ldots,v_n\}$  can be labeled such that the edges of $C$ are $\{v_i,v_{i+1}\}$ for $i=1,\ldots,n-1$ and $\{v_1,v_n\}$. A graph $H$ is called an {\em induced subgraph} of $G$ if there exists a subset $W\subset V(G)$  with $V(H)=W$ and $E(H)=\{\{u,v\}\in E(G)\:\; u,v\in W\}$. The graph $G$ is called {\em chordal} if  any cycle $C$ of $G$ has a chord, where a chord of $C$ is defined to be an edge $\{u,v\}$ of $G$ with $u,v\in V(C)$ but $\{u,v\}\not\in E(C)$. Finally, the graph $Cl$ with $V(Cl)=\{v_1,v_2,v_3,v_4\}$ and $E(Cl)=\{\{v_1,v_2\}, \{v_1,v_3\},\{v_1,v_4\}\}$ is called a {\em claw}, and $G$ is called {\em claw free} if $G$ does not contain an induced subgraph which is isomorphic to $Cl$.

Now we are in the position to formulate the main result of this section.

\begin{Theorem}
\label{chordalclawfree}
Let $G$ be a Koszul graph. Then $G$ is chordal and claw free.
\end{Theorem}

For the proof of this theorem we shall need the following lemma which provides a necessary condition for Koszulness.

\begin{Lemma}
\label{necessary}
Let $S=K[x_1,\ldots,x_n]$ be the polynomial ring over the field $K$ in the variables $x_1,\ldots, x_n$, and let $I\subset S$ be a graded ideal of $S$ generated by quadrics. Denote the graded Betti numbers of $S/I$ by $\beta_{ij}^S(S/I)$ and suppose that $\beta_{2j}^S(S/I)\neq 0$ for some $j>4$. Then $S/I$ is not Koszul.
\end{Lemma}

This lemma is an immediate consequence of Formula (2) given in the introduction of \cite{ACI},  where, as a consequence of results in that paper,  it is stated that if $S/I$ is Koszul, then  $t_{i+1}(S/I)\leq t_i(S/I)+2$ for $i\leq \codim S/I+1$. Here  $t_i(S/I)=\max\{j\:\; \beta_{ij}^S(S/I)\neq 0\}$ for $i=0,\dots, \projdim S/I$.

\medskip
For the convenience of the reader we give a direct proof of the lemma: let $(R,\mm, K)$ be a (Noetherian) local ring or a standard graded $K$-algebra (in which case we assume that $\mm$ is the graded maximal ideal of $R$). Tate in his famous paper \cite{T} constructed an $R$-free resolution
\[
X\: \cdots \To X_i \To \cdots \To X_2\To X_1\To X_0\To 0,
\]
of the residue class field $R/\mm=K$, that is, an acyclic complex of finitely generated free $R$-modules $X_i$ with $H_0(X)=K$,  admitting an additional structure, namely the structure of a differential graded $R$-algebra. It was Gulliksen \cite{G} who proved that if Tate's construction is minimally done, as explained below, then $X$ is indeed a minimal free $R$-resolution of $K$. For details we refer to the original paper of Tate and to a modern treatment of the theory as given in \cite{A}.

Here we sketch Tate's construction as much as is needed to prove the lemma. In Tate's theory $X$ is a DG-algebra, that is,  a graded skew-symmetric $R$-algebra with  free $R$-modules $X_i$ as graded components and $X_0=R$,  equipped with a differential $d$ of degree $-1$ such that
\begin{eqnarray}
\label{product}
d(ab)=d(a)b+(-1)^iad(b)
\end{eqnarray}
for $a\in X_i$ and $b\in X$. Moreover, $(X,d)$ is an acyclic complex with $H_0(X)=K$.

The algebra $X$ is constructed by adjunction of variables: given any DG-algebra $Y$ and a cycle $z\in Y_i$, then the DG-algebra $Y'=Y\langle T\:\, dT=z\rangle $ is obtained  by adjoining the variable $T$ of degree $i+1$ to $Y$ in order to kill the cycle $z$.

If $i$ is even we let
\[
Y_j'=Y_j\dirsum Y_{j-i-1}T \quad \text{with $T^2=0$ and $d(T)=z$.}
\]
If $i$ is odd we  let
\[
Y_j'=X_j\dirsum X_{j-(i+1)}T^{(1)}\dirsum X_{j-2(i+1)}T^{(2)} \dirsum\cdots \quad
\]
with  $T^{(0)}=1$,  $T^{(1)}=T$, $T^{(i)}T^{(j)} = ((i + j)!/i!j!)T^{(i+j)}$ and $d(T^{(i)})=zT^{(i-1)}$. The $T^{(j)}$ are called the divided powers of $T$. The degree of $T^{(j)}$ is defined to be $j\deg T$.

\medskip
The construction of $X$ proceeds as follows:  Say, $\mm$ is minimally generated by $x_1,\ldots, x_n$. Then we adjoin to $R$ (which is a DG-algebra concentrated in homological degree 0)  the variables $T_{11},\ldots,T_{1n}$ of degree 1 with $d(T_{1i})=x_i$. The DG-algebra $X^{(1)}=R\langle T_{11},\ldots,T_{1n}\rangle$ so obtained
is nothing but the Koszul complex of the sequence $x_1,\ldots,x_n$ with values in $R$. If $X^{(1)}$ is acyclic, then $R$ is regular and $X=X^{(1)}$ is the Tate resolution of $K$. Otherwise $H_1(X^{(1)})\neq 0$ and we choose cycles $z_1,\ldots, z_m$ whose homology classes form a $K$-basis of  $H_1(X^{(1)})$, and we adjoin variables $T_{21},\ldots,T_{2m}$ of degree $2$  to  $X^{(1)}$ with $d(T_{2i})=z_i$ to obtain $X^{(2)}$. It is then clear that $H_j(X^{(2)})=0$ for $j=1$.  Suppose $X^{(k)}$ has been already constructed with $H_j(X^{(k)})=0$ for $j=1,\ldots,k-1$. We first observe that $H_k(X^{(k)})$ is annihilated by $\mm$. Indeed, let $z$ be a cycle of $X^{(k)}$, then $x_iz=d(T_{1i}z)$, due to the product rule (\ref{product}).  Now  one chooses a $K$-basis  of cycles representing the homology classes of $H_k(X^{(k)})$ and adjoins variables in degree $k+1$ to kill these cycles, thereby obtaining $X^{(k+1)}$. In this way one obtains a chain of DG-algebras
\[
R=X^{(0)}\subset X^{(1)}\subset X^{(2)} \subset \cdots \subset X^{(2)} \subset\cdots
\]
which in the limit yields the Tate resolution $X$ of $K$. It is clear that if $R$ is standard graded then in each step the representing cycles that need to be killed can be chosen to be  homogeneous, so that $X$ becomes a graded minimal free $R$-resolution of $K$ if we assign to the variables $T_{ij}$ inductively the degree of the cycles they do  kill and apply the following rule:  denote the internal degree (different from the homological degree) of a homogeneous element $a$ of $X$ by $\Deg(a)$. Then we require that $\Deg T^{(i)}=i\Deg T$ for any variable of even homological degree and furthermore  $\Deg(ab)=\Deg(a)+\Deg(b)$ for any two homogeneous elements in $X$.

\medskip
Now we are ready to prove Lemma~\ref{necessary}: the Koszul complex $X^{(1)}$ as a DG-algebra over $S/I$ is generated  by the variable $T_{1i}$ with $d(T_{1i})=x_i$ for $i=1,\ldots,n$. Thus $\Deg T_{1i}=1$ for all $i$. Let $f_1,\ldots,f_m$ be quadrics which minimally generate $I$, and write $f_i=\sum_{j=1}^mf_{ij}x_j$ with suitable linear forms $f_{ij}$. Then $H_1(X^{(1)})$ is minimally generated by the homology classes of the cycles $z_i=\sum_{j=1}^mf_{ij}T_{1j}$. Let $T_{2i}\in X^{(2)}$ be the variables of homological degree $2$ with $d(T_{2i})=z_i$ for $i=1,\ldots,m$. Then $\Deg T_{2i}=\Deg z_i=2$ for all $i$. To proceed in the construction of $X$ we have to kill the cycles $w_1,\ldots,w_r$ whose homology classes form a $K$-basis of  $H_2(X^{(2)})$. Since $\Tor_i(K,S/I)\iso H_i(X^{(1)})$, our hypothesis implies that there is a cycle $z\in (X^{(1)})_2$ with $\Deg z=j>4$ which is not a boundary. Of course $z$ is also a cycle in $X^{(2)}$ because $X^{(1)}$ is a subcomplex of $X^{(2)}$. We claim that $z$ is not a boundary in $X^{(2)}$. To see this we consider the exact sequence of complexes
\[
0\To X^{(1)}\To X^{(2)}\To X^{(2)}/X^{(1)}\To 0,
\]
which induces the long exact sequence
\begin{eqnarray*}
\label{connecting}
\begin{CD}
\cdots @>>>  H_3(X^{(2)}/X^{(1)})@>\delta >> H_2(X^{(1)})\To H_2(X^{(2)})@>>>  \cdots
\end{CD}
\end{eqnarray*}
Thus it suffices to show that the homology class $[z]$ of the cycle $z$ is not in the image of $\delta$. Notice that the elements $T_{1i}T_{2j}$ form a basis of the free $S$-module $(X^{(2)}/X^{(1)})_3$ and that the differential on  $X^{(2)}/X^{(1)}$ maps $T_{1i}T_{2j}$ to $x_iT_{2j}$, so that $w\in (X^{(2)}/X^{(1)})_3$ is a cycle if and only if $w=\sum_{j=1}^mw_jT_{2j}$ where each $w_j\in X^{(1)}_1$  is a cycle. Now the connecting homomorphism $\delta$ maps $[w]$ to $[-\sum_{j=1}^mw_jz_j]$. It follows that $\Im \delta=H_1(X^{(1)})^2$. Since $H_1(X^{(1)})$ is generated in degree $2$ we conclude that the subspace $H_1(X^{(1)})^2$ of $H_2(X^{(1)})$  is generated in degree $4$. Hence our element $[z]\in H_2(X^{(1)})$ which is of degree $>4$ cannot be in the image of $\delta$, as desired.

Thus the homology class of $z$,  viewed as an element of $H_2(X^{(2)})$ has to be killed by adjoining a variable a variable of degree $j>4$. This shows that $\beta_{3j}^{S/I}(S/\mm)\neq 0$, and hence $S/I$ is not Koszul.

\begin{proof}[Proof of Theorem~\ref{chordalclawfree}]
We may assume that $[n]$ is the vertex set of $G$. Let  $H$ by  any induced subgraph of $G$. We may further assume that $V(H)=[k]$.  Let $S=K[x_1,\ldots,x_n,y_1,\ldots,y_n]$ and $T=K[x_1,\ldots,x_k,y_1,\ldots,y_k]$. Then $T/J_H$ is an algebra retract of $S/J_G$. Indeed, let $L=(x_{k+1},\ldots,x_n,y_{k+1},\ldots,y_n)$. Then the composition $T/J_H\to S/J_G\to S/(J_G, L)\iso T/J_H$ of the natural $K$-algebra homomorphisms is an isomorphism. It follows therefore from \cite[Corollary 2.6]{OHH} that any induced subgraph of a $G$ is again Koszul.

Suppose  that $G$ is not claw free. Then there exists an induced subgraph $H$ of $G$ which is isomorphic to a claw. We may assume that $V(H)=\{1,2,3,4\}$, and let $R=K[x_1,\ldots,x_4,y_1,\ldots,y_4]$. A computation with Singular \cite{DGPS} shows that $\beta^{R/J_H}_{3,5}(K)\neq 0$. Thus $H$ is not Koszul, a contradiction.

Suppose that $G$ is not chordal. Then there exist a cycle $C$ of length $\geq 4$ which has no chord. Then $C$ is an induced subgraph and hence should be Koszul. We may assume that $V(C)=\{1,2,\ldots,m\}$ with edges $\{i,i+1\}$ for $i=1,\ldots,m-1$ and edge $\{1,m\}$ and set $T=K[x_1,\ldots,x_m,y_1,\ldots,y_m]$. We claim that $\beta_{2,m}^T(T/J_C)\neq 0$. For $m>4$  this will imply that $C$ is not Koszul. That a $4$-cycle is not Koszul can again be directly checked with Singular \cite{DGPS}.

In order to prove the claim  we let $F=\Dirsum_{i=1}^mS e_{i}$ and consider the free presentation
\[
\epsilon\:\; F\to I\To 0,\quad \text{$e_{i}\mapsto f_{i,i+1}$ for $i=1,\ldots,m$}
\]
For simplicity, here and in the following,   we read $m+1$ as $1$.

 Obviously, $g=\sum_{i=1}^m(\prod_{j=1}^mx_j)/(x_ix_{i+1})e_i\in\Ker \epsilon$. We will show that $g$ is a minimal generator
of $\Ker \epsilon$. Indeed, let $g'=\sum_{i=1}g_ie_i\in \Ker \epsilon$ be an arbitrary relation, and suppose that some $g_j=0
$. Since the $f_{i,i+1}$ for $i\neq j$ form a regular sequence, it then follows that all the other $g_i$ belong to $J_C$.
However, since the coefficients of $g$ do not belong to $J_C$, we conclude that $g$ cannot be written as a linear
combination of relations for which one of its  coefficients is zero.

 Now assume that all $g_i\neq 0$. Let  $\epsilon_i$ denotes the $i$th canonical unit vector of $\ZZ^n$.  Since $J_C$ is a $
\ZZ^n$-graded ideal with $\deg_{\ZZ^n} x_i= \deg_{\ZZ^n} y_i=\epsilon_i$,  we may assume that  $g^\prime=\sum_{i=1}g_ie_i$ is a
homogeneous relation where  $\deg_{\ZZ^n} e_i=\deg f_{i,i+1}=\epsilon_i+\epsilon_{i+1}$ and $g_i$ is homogeneous satisfying
$\deg_{\ZZ^n} g'= \deg_{\ZZ^n} g_i+\epsilon_i+\epsilon_{i+1}$ for all $i$. This is only possible if $\deg_{\ZZ^n} g'\geq \sum
_{i=1}^m \epsilon_i$, coefficientwise. In particular it follows that $\deg g'\geq m$, where $\deg g'$ denotes the total
degree of $g'$. Thus $g$ cannot be a linear combination of relations of lower (total) degree and hence is a minimal
generator of $\Ker \epsilon$. Since $\deg g=m$, we conclude that $\beta^T_{2,m}(T/J_C)\neq 0$.
\end{proof}

\section{Gluing of Koszul graphs along a vertex}
\label{two}

In this section we first show that Koszulness is preserved under the operation of gluing two graphs along a vertex in the sense
that we are going to explain below.

We begin with two general statements about Koszul algebras.

\begin{Proposition}\label{tensor}
Let $R=K[x_1,\ldots,x_n]/I$ and $S=K[x_{n+1},\ldots,x_m]/J$ be two standard graded $K$-algebras. Then $R\otimes_K S$ is
Koszul if and only if $R$ and $S$ are Koszul.
\end{Proposition}

\begin{proof}
Let $\mm$ and $\nn$ be the maximal ideals of $R$ and, respectively, $S.$ Let $\FF\to R/\mm\to 0$ and
$\GG\to S/\nn\to 0$ be the minimal graded free resolutions of $R/\mm$ over $R$ and, respectively, of $S/\nn$ over $
S.$ Then the total complex of $\FF\otimes \GG$ is the minimal graded free resolution over $R\otimes S$
of the maximal graded ideal of $R\otimes S.$ If $F_i=\bigoplus_{k}R(-k)^{\beta_{ik}}$ for all $i$ and
$G_j=\bigoplus_\ell S(-\ell)^{\beta^\prime_{j \ell}}$ for all $j,$ then $$F_i\otimes G_j\cong \bigoplus_{k,\ell} R\otimes S(-k-\ell)^{\beta_{ik}\beta^\prime_{j\ell}}.$$ This
immediately implies the desired conclusion  if $\FF$ and $\GG$ are linear. For the converse, we note that
$\Tor_p^{R\otimes S}(K,K)\cong \bigoplus_{{i+j}=p}\bigoplus_{k,\ell}K(-k-\ell)^{\beta_{ik}\beta^\prime_{j\ell}}$. Then we must
have $k+\ell=p$ for all $i,j$ with $i+j=p.$ As $k\geq i$ and $\ell\geq j,$ it follows that $k=i$ and $\ell=j$ for all $i,j.$
Therefore, $\FF$ and $\GG$ are linear resolutions as well.
\end{proof}

The above proposition shows, in particular, that it is enough to study the Koszul property for connected graphs.

\begin{Corollary}\label{connected}
Let $G$ be a graph with connected components $G_1,\ldots,G_r.$ Then $G$ is Koszul if and only if $G_i$ is Koszul for  $1\leq
i\leq r.$
\end{Corollary}

\begin{proof}
Let $V(G)=[n]$ and $S=K[x_1,\ldots,x_n,y_1,\ldots,y_n].$ Then $S/J_G\cong\otimes_{i=1}^r S_i/J_{G_i}$ where $S_i=K[\{x_j,y_j: j\in V(G_i)\}]$ for $1\leq i\leq r.$ The
claim follows  by applying Proposition~\ref{tensor}.
\end{proof}

\begin{Proposition}\label{modreg}
Let $R$ be a standard graded $K$-algebra with maximal graded ideal $\mm$ and $f_1,\ldots,f_m\in \mm\setminus \mm^2$  a regular
sequence of homogeneous elements in $R.$ Then $R$ is Koszul if and only if $R/(f_1,\ldots,f_m)$ is Koszul.
\end{Proposition}

\begin{proof}
By induction on $m,$ it is sufficient to prove the claim for $m=1.$ Let then $f\in R$ be a form of degree $1.$ We have to
show that $R$ is Koszul if and only if $R/(f)$ is Koszul or, equivalently, $K$ has a linear resolution over $R$ if and only
if it has a linear resolution over $R/(f)$. But this is  a direct consequence of \cite[Theorem 2.2.3]{A}.
\end{proof}

Now we come to the main subject of this section. By Corollary~\ref{connected}, in the sequel we may assume
that all the graphs are connected.

Let $G$ be a graph. A {\em clique} of $G$ is a complete subgraph of $G.$ The cliques
of $G$ form a simplicial complex $\Delta(G)$ which is called the {\em clique complex} of $G.$ The facets of $\Delta(G)$ are
the maximal cliques of $G$ with respect to inclusion. A {\em free} vertex of $\Delta(G)$ or, simply, of $G$ is a vertex of $G$ which belongs only to one facet of $\Delta(G).$

Let $G_1,G_2$ be to graphs such that $V(G_1)\cap V(G_2)=\{v\}$ and $v$ is a free vertex in $G_1$ and $G_2.$ Let $G=G_1\cup G_
2$ with $V(G)=V(G_1)\cup V(G_2)$ and $E(G)=E(G_1)\cup E(G_2)$. We say that $G$ is obtained by {\em gluing $G_1$ and $G_2$
along the vertex $v$}.

\begin{Theorem}\label{gluing}
Let $G$ be a graph obtained by gluing the graphs $G_1$ and $G_2$ along a vertex. Then $G$ is Koszul if and only if $G_1$ and
$G_2$ are Koszul.
\end{Theorem}

\begin{proof}
Let $V(G)=[n]$ and assume that $G_1$ and $G_2$ are glued along the vertex $v\in [n].$ Let $v^\prime$ be a vertex which does
not belong to $V(G)$ and let $G_2^\prime$ be the graph with $V(G_2^\prime)=(V(G_2)\setminus \{v\})\cup\{v^\prime\}$ whose
edge set is $E(G_2^\prime)=E(G_2\setminus\{v\})\cup\{\{i,v^\prime\}:\{i,v\}\in E(G_2)\}$. We set $S=K[x_1,\ldots,x_n,y_1,\ldots,y_n]$ and $S^\prime=S[x_{v^\prime},y_{v^\prime}].$ Let $\ell_{x}=x_{v}-x_{v^\prime}$ and $\ell_{y}=y_v-y_{v^\prime}.$
By the proof of \cite[Theorem 2.7]{RR}, we know that $\ell_x,\ell_y$ is a regular sequence on $S^\prime/J_{G^\prime}$, where
$G^\prime$ is the graph whose connected components are $G_1$ and $G_2^\prime.$ Moreover, we obviously have
\[
S^\prime/(J_G^\prime,\ell_x,\ell_y)\cong S/J_G.
\]
By Proposition~\ref{modreg}, it follows that $G$ is Koszul if and only if $G^\prime$ is Koszul. Next,
by Corollary~\ref{connected}, it follows that $G^\prime$ is Koszul if and only its connected components, namely  $G_1$ and
$G_2^\prime$, are Koszul. Finally, we observe that $G^\prime_2$ is Koszul if and only if $G_2$ is so.
\end{proof}

Let $G$ be a graph.  By Dirac's theorem \cite{D}, $G$ is chordal if and only if the facets of $\Delta(
G)$ can be ordered as $F_1,\ldots,F_r$ such that, for all $i>1,$ $F_i$ is a leaf of the simplicial complex
$\langle F_1,\ldots,F_{i-1}\rangle$. This means that there exists a facet $F_j$ with $j<i$ which intersects $F_i$ maximally,
that is, for each $\ell<i,$ $F_\ell\cap F_i\subset F_j\cap F_i$. $F_j$ is called a {\em branch} of $F_i.$

The following corollary gives a class of chordal and claw-free graphs which are Koszul.

\begin{Corollary}\label{corchordalclawfree}
Let $G$ be a chordal and claw-free graph with the property that  $\Delta(G)$ admits a leaf order  $F_1,\ldots,F_r$ such that  for all $i>1,$ the facet $F_i$ intersects any of  its branches in one vertex. Then $G$ is Koszul.
\end{Corollary}

\begin{proof}
We proceed by induction on $r.$ If $r=1$, there is nothing to prove since any clique is Koszul. Let $r>1$ ans assume
that the graph $G^\prime$ with $\Delta(G^\prime)=\langle F_1,\ldots, F_{r-1}\rangle$ is Koszul. We may assume that $F_{r-1}$
is a branch of $F_r$ and let $\{v\}=F_r\cap F_{r-1}.$  The desired statement follows by applying Theorem~\ref{gluing} for $G^
\prime$ and the clique $F_r$, once we show that $v$ is a free vertex of $G^\prime.$

Let us assume that $v$ is not free in $G^
\prime$ and choose a maximal clique $F_j$ with $j\leq r-2$ such that $v\in F_j$.  We may find three vertices $a,b,c\in V(G)$
such that $a\in F_r\setminus (F_{r-1}\cup F_j),$ $b\in F_{r-1}\setminus(F_r\cup F_j)$, and $c\in F_j\setminus(F_r\cup F_{r-1}
).$ If $\{a,b\}\in E(G)$, then there exists a maximal clique $F_k$ with $k\leq r-1$ such that $a,b\in F_k.$ This implies that
$a\in F_k\cap F_r\subset \{v\},$ contradiction. Therefore, $\{a,b\}$ is not an edge of $G$. Similarly, one proves that
$\{a,c\}\notin E(G)$. Let us now assume that $\{b,c\}\in E(G).$ The clique on the vertices $v,b,c$ is contained in some
maximal clique $F_k$. We have $k\leq r-2$ since $F_k\neq F_{r-1}.$ Then it follows that $|F_k\cap F_{r-1}|\geq 2$ which is a contradiction to our hypothesis on $G.$ Consequently, we have proved that $\{a,b\}, \{b,c\}, \{a,c\}\notin E(G).$ Hence, $G$ contains a claw as an induced subgraph, contradiction. Therefore, $v$ is a free vertex of $G^\prime.$
\end{proof}

In Figure~\ref{example} is shown   a graph which satisfies the conditions of Corollary~\ref{corchordalclawfree} and  is not closed.

\begin{figure}[hbt]
\begin{center}
\psset{unit=0.8cm}
\begin{pspicture}(1,1)(5,5)
\pspolygon(2,2)(3,3.71)(4,2)
\psline(3,3.71)(3,5.2)
\psline(0.6,1.1)(2,2)
\psline(4,2)(5.6,1.1)
\rput(2,2){$\bullet$}
\rput(3,3.71){$\bullet$}
\rput(4,2){$\bullet$}
\rput(3,5.2){$\bullet$}
\rput(0.6,1.1){$\bullet$}
\rput(5.6,1.1){$\bullet$}
\end{pspicture}
\end{center}
\caption{}
\label{example}
\end{figure}

Figure~\ref{notkoszul} displays  a chordal and claw free graph $G$ which is not Koszul. That  $G$  is not Koszul can be seen as follows: we first observe that the graph $G'$ restricted to the vertex set $[4]$ is Koszul by Corollary~\ref{corchordalclawfree}, and that $B=K[x_1,\ldots,x_4,y_1,\ldots,y_4]/J_{G'}$ is an algebra retract of $A=K[x_1,\ldots,x_6,y_1,\ldots,y_6]/J_G$ with retraction map $A\to A/(x_5,x_6,y_5,y_6)\iso B$. Thus if $A$ would be Koszul, the ideal $(x_5,x_6,y_5,y_6)$ would have to have an $A$-linear resolution, see \cite[Proposition 1.4]{OHH}. It can be verified with Singular \cite{DGPS} that this is not the case.

\begin{figure}[hbt]
\begin{center}
\psset{unit=1.5cm}
\begin{pspicture}(0,0.5)(3,2.5)
\pspolygon(0.5,0.5)(1.5,0.5)(1,1.4)
\pspolygon(1.5,0.5)(2,1.4)(2.5,0.5)
\pspolygon(1,1.4)(1.5,2.3)(2,1.4)

\rput(0.5,0.5){$\bullet$}
\rput(1.5,0.5){$\bullet$}
\rput(1,1.4){$\bullet$}
\rput(2,1.4){$\bullet$}
\rput(2.5,0.5){$\bullet$}
\rput(1.5,2.3){$\bullet$}

\rput(0.3,0.3){$3$}
 \rput(1.5,0.3){$5$}
 \rput(0.8,1.4){$2$}
 \rput(2.2,1.4){$4$}
 \rput(2.7,0.3){$6$}
 \rput(1.5,2.5){$1$}

\end{pspicture}
\end{center}
\caption{}\label{notkoszul}
\end{figure}

\medskip
A {\em line graph} of length $m$ is a graph which is isomorphic  to the  graph with edges $\{1,2\},\{2,3\},\ldots,\{m-1,m\}$.  A {\em $2$-dimensional line graph} is a graph whose cliques are $2$-dimensional cliques composed as shown in  Figure~\ref{twodimensional}.

\begin{figure}[hbt]
\begin{center}
\psset{unit=1.5cm}
\begin{pspicture}(1,0.5)(4,2.5)
\pspolygon(0.5,0.5)(1.5,0.5)(1,1.4)
\pspolygon(1.5,0.5)(2,1.4)(2.5,0.5)
\psline(1,1.4)(2,1.4)
\pspolygon(2.5,0.5)(3,1.4)(3.5,0.5)
\psline(2,1.4)(3,1.4)
\pspolygon(3.5,0.5)(4,1.4)(4.5,0.5)
\psline(3,1.4)(4,1.4)

\rput(0.5,0.5){$\bullet$}
\rput(1.5,0.5){$\bullet$}
\rput(1,1.4){$\bullet$}
\rput(2,1.4){$\bullet$}
\rput(2.5,0.5){$\bullet$}
\rput(3,1.4){$\bullet$}
\rput(4,1.4){$\bullet$}
\rput(3.5,0.5){$\bullet$}
\rput(4.5,0.5){$\bullet$}

\end{pspicture}
\end{center}
\caption{}\label{twodimensional}
\end{figure}

To be precise,  a $2$-dimensional line graph of length $m$ is a graph whose clique complex is isomorphic to the simplicial complex with  facets
\[
\{1,2,3\},\{2,3,4\}, \ldots,\{m-1,m,m+1\}, \{m,m+1,m+2\}.
\]

From what we have shown so far it is not too difficult to obtain the following classification result, which roughly says that any connected Koszul graph  whose clique  complex is of dimension $\leq 2$ is obtained by gluing $1$-dimensional  and $2$-dimensional line graphs.

\begin{Theorem}
\label{classification}
Let $G$ be a connected graph whose clique  complex is of dimension $\leq 2$. The following conditions are equivalent:
\begin{enumerate}
\item[(a)] $G$ is Koszul;
\item[(b)] There exists a tree $T$ whose vertices have order at most $3$ such that $G$ is obtained from $T$ as follows:
\begin{enumerate}
\item[(i)] each vertex $v$ of $T$ is replaced by a $1$-dimensional or $2$-dimensional line graph $G_v$;
\item[(ii)] if $\{v,w\}$ is an edge of $T$ then $G_v$ and $G_w$ are glued via a free vertex of $G_v$ and $G_w$;
\item[(iii)] if $v$ is a vertex of order $3$ and $w_1,w_2,w_3$ are the  neighbors of $v$, then $G_v$ is a simplex and each $G_{w_i}$ is glued to a different vertex of $G_v$.
\end{enumerate}
\end{enumerate}
\end{Theorem}

\begin{proof} (a)\implies (b): Let $D$ be a subcomplex  consisting of $2$-dimensional facets of $\Delta(G)$, and assume that $D$  is  connected in codimension $1$. By that we mean that, for any two facets $F, F'\in D$, there exist facets $F_1,\ldots,F_r$ such that $F=F_1$ and $F'=F_r$, and such that $F_i$ and $F_{i+1}$ intersect along an edge for $i=1,\ldots,r-1$. We claim that the $1$-skeleton $H$ of $D$ is a $2$-dimensional line graph, and prove this by induction on the number of facets  of $D$. The assertion is trivial if the number of facets of $D$ is $\leq 3$, because since $G$ is claw free,  $D$ cannot be isomorphic to the graph with edges $\{1,2,3\},\{1,2,4\}$ and $\{1,2,5\}$. Now assume that $D$ has $m+1>3$ facets. By Theorem~\ref{chordalclawfree}, the graph $G$ is chordal, and hence $H$ is chordal as well. Applying Dirac's theorem \cite{D} we conclude that $D$ admits a leaf $F$. Let $D'$ be the subcomplex of $D$ which is obtained from $D$ by removing the leaf $F$. Then $D'$ is again connected in codimension $1$. Our induction hypothesis implies that the $1$-skeleton $H'$ of $D'$ is a $2$-dimensional line graph. For simplicity we may assume that the facets of $D'$ are $\{1,2,3\},\{2,3,4\}, \ldots,\{m-1,m,m+1\}, \{m,m+1,m+2\}$. If $F=\{a,1,2\}$ or $F=\{m+1,m+2, b\}$, then $H$ is a $2$-dimensional line graph, and we are done. Otherwise, $F=\{i,i+1,c\}$ or $F=\{i,i+2,c\}$ for some $i\in [m]$ and some vertex $c$ of $D$.
The first case cannot happen, since $G$ is claw free. In the second case,
if $1<i<m$, then $D$, and consequently, $\Delta(G)$ contains an induced subgraph isomorphic to the graph in Figure~\ref{notkoszul}. Thus $G$ is not Koszul, a contradiction. On the other hand, if $i=1$, then $H$ is not claw free, because then edges  $\{3,c\},\{3,2\},\{3,5\}$ form a claw which is an induced subgraph of $G$, contradiction to the fact that $G$ must be claw free.

Let $D_1,\ldots, D_r$ be  the maximal $2$-dimensional subcomplexes of $\Delta(G)$ which are connected in codimension $1$, and $L_1,\ldots, L_s$ be the maximal $1$-dimensional connected subcomplexes of $\Delta(G)$. Each $L_i$ is a $1$-dimensional line graph, otherwise $G$ would not be claw free, and the $D_i$ are all $2$-dimensional line graphs, as we have seen above. The maximality of the $L_i$  implies that $V(L_i)\sect V(L_j)=\emptyset$ for $i\neq j$, and the
maximality of  $D_i$ implies that each facet of $D_i$ intersects any facet of $\Delta(G)$ not belonging to $D_i$ in at most one vertex.

Now we let $T$ be the graph whose vertices are $D_1,\ldots,D_r, L_1,\ldots,L_s$.  The edge set $E(T)$ consists of the edges  $\{D_i, D_j\}$  if $V(D_i)\sect V(D_j)\neq \emptyset$ and $\{D_i,L_j\}$ if $V(D_i)\sect V(L_j)\neq  \emptyset$.

If $\{D_i, D_j\}\in E(T)$ and $v$ is a common vertex of $D_i$ and $D_j$, then $v$ must be a free vertex of $D_i$ and of $D_j$, because otherwise $G$ would not be claw free. Moreover, $|V(D_i)\sect V(D_j)|\leq 1$,  because otherwise $G$ contains a cycle of length $>3$ without chord, contradicting the fact that $G$ is chordal. By the same reason we have that $|V(D_i)\sect V(L_j)|\leq 1$ for all $i$ and $j$.

Next observe that the intersection of any three of the sets $$V(D_1),\ldots, V(D_r), V(L_1),\ldots, V(L_s)$$ is the empty set, which follows from the fact that $G$ is claw free. Thus the order of the vertices of $T$ can be at most the number of free vertices of an $D_i$ or $L_j$, and hence is at most $3$, where the maximal number $3$ can be reached only if one of the $D_i$ is a $2$-simplex. Finally $T$ must be a tree, because otherwise $G$ would not be chordal.

(b)\implies (a): We proceed by induction on $V(T)$. If $V(T)=1$, then $G$ is a $1$-dimensional or $2$-dimensional line graph. In both cases $G$ is closed and hence has a quadratic Gr\"obner basis. This implies that $G$ is Koszul. Now let $V(T)>1$, and choose a free vertex $v\in V(T)$. Then $T'=T\setminus v$  is again a tree satisfying  conditions~(b). Let $W=\Union_{w\in T'}V(G_w)$ and $G'$ the restriction of $G$ to $W$. Then our induction hypothesis implies that $G'$ is Koszul. Since $G_v$ is Koszul and since $G$ is obtained from $G'$ and $G_v$ by gluing along a common free vertex, the desired conclusion follows from Theorem~\ref{gluing}.
\end{proof}

\begin{Remark}
\label{infinite}{\em
Let $G$ be a connected Koszul graph whose clique  complex is of dimension $\leq 2$, and let $T$ be its ``intersection tree" as described in Theorem~\ref{classification}(b). If $\Delta(G)$ does not contain a subcomplex as given in Figure~\ref{example}, then $T$ is a line graph. In this case $G$ is obtained by gluing in alternative order $1$-dimensional and $2$-dimensional line graphs. Thus it follows from \cite[Theorem 2.2]{EHH} that $G$ is closed and hence $J_G$  has a quadratic Gr\"obner basis.

On the other hand, if $T$ contains a subcomplex as shown in Figure~\ref{example}, then $G$ is not closed, and hence by a result of Crupi and Rinaldo \cite[Theorem 3.4]{CR} it follows that $J_G$ has no quadratic Gr\"obner basis for any monomial order.  Thus those graphs provide and infinite family of binomial ideals which define a Koszul algebra but do not have a quadratic Gr\"obner basis.

Examples of toric rings which are Koszul but do not have a  quadratic Gr\"obner basis were found independently by Roos and Sturmfels \cite{RS} and by Ohsugi and Hibi \cite{OH}.
}
\end{Remark}

\end{document}